\documentclass{article}
\usepackage{amsthm}
\usepackage{amsmath,amssymb}
\usepackage{helvet}
\usepackage{docmute}
\usepackage{tcolorbox}
\usepackage[enableskew]{youngtab}
\usepackage{pxrubrica}
\usepackage{color}
\usepackage{tikz}
\usepackage{ytableau}
\usepackage{braket}
\usepackage{autobreak}
\usepackage{mathtools}
\usepackage{hyperref}
\usepackage{endnotes}
\usepackage{authblk}
\hypersetup{%
 setpagesize=false,%
 bookmarks=true,%
 bookmarksdepth=tocdepth,%
 bookmarksnumbered=true,%
 colorlinks=false,%
 pdftitle={},%
 pdfsubject={},%
 pdfauthor={},%
 pdfkeywords={}}
\theoremstyle{definition}
\newtheorem{dfn}{Definition}[section]
\newtheorem{thm}{Theorem}[section]

\newtheorem{prop}{Proposition}[section]
\newtheorem{cor}{Corollary}[section]

\theoremstyle{definition}
\newtheorem{case}{Example}[section]
\newtheorem*{prop*}{proposition}
\numberwithin{equation}{section}

\newtheorem{rem}{Remark}[section]
\theoremstyle{plain}

\allowdisplaybreaks

\makeatletter
\renewcommand{\AB@affilsep}{\quad\protect\Affilfont}
\let\AB@affilsepx\AB@affilsep 

\makeatother

\begin{document}

\title{Special values of Grothendieck polynomials in terms of hypergeometric functions}

\author{Taikei Fujii, Takahiko Nobukawa and Tatsushi Shimazaki}

\date{\empty}

\maketitle

\begin{abstract}
We give some special values of Grothendieck polynomials and an explicit formula for the number of set-valued tableaux. For Young diagrams consisting of a single row or a single column, both the value and number are written by the Gauss' hypergeometric function ${}_2F_1$. For general Young diagrams, the Holman hypergeometric function $F^{(n)}$ is used to represent both the value and count. As an application, we derive a summation formula for $F^{(n)}$. 
\end{abstract}
\renewcommand{\thefootnote}{\fnsymbol{footnote}}
\footnote[0]{\noindent Keywords: set-valued tableau, Grothendieck polynomial, hypergeometric function.}
\footnote[0]{MSC 2020: 05A15, 05A17, 05E05, 33C70, 33C80.}
\renewcommand{\thefootnote}{\arabic{footnote}}

\section{Introduction}\label{I}
We consider $\lambda$ as a positive integer partition and identify it with the Young diagram. 
As well known, the Schur polynomials $s_\lambda$ can be expressed in terms of semi-standard tableaux. 
We use the notation ${\rm SST}(\lambda,n)$ to represent the set of semi-standard tableaux with shape $\lambda$ and consisting of $n$ variables. The value $|{\rm SST}(\lambda,n)|$ is calculated by the hook-length formula~\cite{[Mac95],[Nou23]}:
\begin{align}
|{\rm SST}(\lambda,n)| &= \prod_{1\leq i<j\leq n}\frac{\lambda_i-\lambda_j+j-i}{j-i} \label{h1} \\ 
&=\prod_{(i,j) \in \lambda}\frac{n+j-i}{h_{i,j}} \label{h2},
\end{align}
where $h_{i,j}$ represents the hook-length of each box $(i,j)$ in the Young diagram. 

Set-valued tableaux are one of the generalizations of semi-standard tableaux. These are introduced by Buch~\cite{[Buc02]} to express the tableaux sum formula for stable Grothendieck polynomials $G_\lambda$. Throughout this paper, since we only discuss stable Grothendieck polynomials $G_\lambda$, we refer to $G_\lambda$ as Grothendieck polynomials.
The Grothendieck polynomials were introduced by Lascoux and Sch\"{u}tzenberger~\cite{[Las90],[LS82]} as representatives for the structure sheaves of Schubert varieties in a flag variety.
These polynomials are generalizations of the Schur polynomials $s_\lambda$ in the $K$-theory of Grassmannians.

We denote the set of set-valued tableaux shape $\lambda$ having $n$ variables by ${\rm SVT}(\lambda,n)$, and we consider the number $|{\rm SVT}(\lambda,n)|$.
Unlike the case of $|{\rm SST}(\lambda,n)|$, there seems to be no factorized formula for $|{\rm SVT}(\lambda,n)|$ as~(\ref{h2}).
For example, we have
\begin{align*}
&|{\rm SVT}((2,1),3)|=27,\ |{\rm SVT}((2,2),3)|=13,\ |{\rm SVT}((4,3),3)|=103,\\
&|{\rm SVT}((2,1),4)|=159,\ |{\rm SVT}((2,2),4)|=97,\ |{\rm SVT}((4,3),4)|=1759.
\end{align*}
Therefore, we cannot expect a simple factorized formula. 
In this sense, counting $|{\rm SVT}(\lambda,n)|$ is not so easy and is an interesting problem. 
However, an explicit formula for the number $|{\rm SVT}(\lambda,n)|$ was unknown.

In this paper, we give some special values of Grothendieck polynomials.
Using these formulas, we derive an explicit formula for $|{\rm SVT}(\lambda,n)|$ that is similar to the structure of~(\ref{h1}). 
The explicit formula is rewritten by Holman's hypergeometric functions $F^{(n)}$. 
The functions $F^{(n)}$ were introduced by Holman~\cite{Hol80} which appear in the representation theory of the Lie group $U(n+1)$ or $SU(n+1)$.
For special $\lambda$, the number $|{\rm SVT}(\lambda,n)|$ is also given by the Gauss' hypergeometric functions ${}_2F_1$. For the definitions of $G_\lambda,\ {}_2F_1$, and $F^{(n)}$, see~(\ref{GG}),~(\ref{hgf}) and~(\ref{c2}), respectively.

Our main result is as follows: 

\begin{flushleft}
{\bf Main\ result}\ ({\bf Theorem~\ref{M2}}){\bf .}$\quad$ For any $\lambda \vdash l$, we have
\begin{align}\label{G_H}
&G_{\lambda}(1, 1, \dots, 1\mid \beta) \notag \\
&= |{\rm SST}(\lambda,n)|\cdot F^{(n)} \left(\begin{pmatrix}
  A_{12} &\ & & \\
  A_{13} & A_{23} & & \text{\Huge $0$} \\
  \vdots & \vdots &\ddots & \\
  A_{1n} & A_{2n} &\cdots & A_{n-1n}
\end{pmatrix} \middle| \begin{pmatrix}
  0 \\
  -1 \\
  \vdots \\
  -(n-1)
\end{pmatrix} \middle| \begin{pmatrix}
  1 \\
  1 \\
  \vdots \\
  1
\end{pmatrix} \middle| \begin{pmatrix}
  -\beta \\
  -\beta \\
  \vdots \\
  -\beta
\end{pmatrix} \right),
\end{align}
where $A_{ij}=\lambda_i+\lambda_j+j-i$.
\end{flushleft}
As a corollary of this result, we give an explicit formula for the number of set-valued tableaux (see also Corollary~\ref{cor2}).

The organization of this paper is as follows. In section~\ref{S2}, we give the definitions addressed in this paper. In section~\ref{S3}, we prove our main result. 
Additionally, we show that for Young diagrams consisting of a single row or a single column, the number of set-valued tableaux is written by Gauss' hypergeometric function ${}_2F_1$. As an application, we derive a summation formula for $F^{(n)}$. 

\section{Preliminaries}\label{S2}
In this section, we define the subjects addressed in this paper.

\subsection{Partitions and Young diagrams}
Let $l$ be a positive integer. A {\it partition} of $l$ is a weakly decreasing non-negative integer sequence
\begin{equation*}
\lambda = (\lambda_1,\lambda_2,\dots,\lambda_r),
\end{equation*} 
and satisfying
\begin{equation*}
\sum_{i=1}^r \lambda_i = l.
\end{equation*}
We denote a partition of $l$ as $\lambda \vdash l$. We consider $(\lambda_1,\dots,\lambda_r)$ and $(\lambda_1,\dots,\lambda_r,0)$ as equivalent. We define the {\it length} of $\lambda$ to be the smallest number $r$ such that $\lambda_r>0$ and $\lambda_{r+1}=0$. We write $r=\ell(\lambda)$. For $\lambda \vdash l$, the set 
\begin{align*}
\{ (i,j) \in (\mathbb{Z}_{>0})^2\mid1 \leq i \leq \ell(\lambda),\ 1\leq j \leq \lambda_i \}
\end{align*}
is called the {\it Young diagram of shape} $\lambda$. We use the notation that the Young diagram of shape $\lambda$ is obtained by arranging $l$ boxes with $\ell(\lambda)$ left-justified rows with the $i$-th row consisting of $\lambda_i$ boxes. We use the same method to describe the boxes of a Young diagram as matrices. This means that the first coordinate $i$ increases downwards, and the second coordinate $j$ increases from left to right. Since there is a one-to-one correspondence between a partition of $l$ and a Young diagram of shape $\lambda$, we consider them the same. The notation $|\lambda|$ represents the total number of boxes in a Young diagram and is calculated by summing up the number of boxes in each row.

\subsection{Semi-standard tableaux and Schur polynomials $s_\lambda$}
Let $\lambda \vdash l$ and $[n]=\{ 1,2,\dots,n \}$. \textit{Semi-standard tableaux of shape $\lambda$} are diagrams obtained by assigning a non-empty subset consisting of one element of $[n]$ from each $(i,j)$-th box of $\lambda$, satisfying the following conditions:
\begin{itemize}
\item The number in the box $(i,j+1)$ weakly increases more than the number in $(i,j)$.
\item The number in the box $(i+1,j)$ strictly increases more than the number in $(i,j)$.
\end{itemize}
We represent semi-standard tableaux of shape $\lambda$ by the notation ${\rm SST}(\lambda)$. When emphasizing the dependence on the number of variables $n$, we use the notation ${\rm SST}(\lambda, n)$. It is widely known that Schur polynomials $s_\lambda$ have the tableaux sum formula:
\begin{align*}
s_\lambda=s_\lambda(x_1.x_2,\dots,x_n)=\sum_{T\in {\rm SST}(\lambda,n)}x^{\omega(T)},
\end{align*}
where $x=(x_1,x_2,\dots,x_n)\in \mathbb{C}^n$, and $x^{\omega(T)}$ is defined as (\ref{mono}) below.

\subsection{Set-valued tableaux and Grothendieck polynomials $G_\lambda$}\label{SVT}
We denote $b_{i,j}$ as the box at the position $(i,j)$ in a Young diagram $\lambda$. Let $T_{i,j}$ represent a non-empty subset of $[n]$. 
\begin{dfn}~(\cite{[Buc02]})\label{svt}
For a Young diagram $\lambda$, the {\it set-valued semi-standard tableau of shape $\lambda$} is defined as a mapping from the set consisting of the whole boxes $b_{i,j}$ of $\lambda$ to the set consisting of the whole $T_{i,j}$ satisfying the following conditions:
\begin{itemize}
\item ${\rm max}T_{i,j} \leq {\rm min}T_{i,j+1}$.
\item ${\rm max}T_{i,j} < {\rm min}T_{i+1,j}$.
\end{itemize}
\end{dfn}
We call set-valued semi-standard tableaux plainly as set-valued tableaux. 
We write a set of set-valued tableaux of shape $\lambda$ as ${\rm SVT}(\lambda)$. If we emphasize the number of variables $n$, we denote ${\rm SVT}(\lambda,n)$. 
We only consider situations where ${\rm SVT}(\lambda,n)\neq \emptyset$. 
\begin{case}\label{ex2.1}
Let $\lambda = (2,1) \vdash 3$, $n=3$. We take a $T \in {\rm SVT}((2,1),3)$ as follows: \\
\begin{equation*}
{\raisebox{-5.5pt}[0pt][0pt]{$T$\ =\ }} {\raisebox{-2.5pt}[0pt][0pt]{\ytableaushort{{1}{1,\!2},{2,\!3}}}},
\end{equation*} \\ \\
where $\scriptsize{{\raisebox{-2.5pt}[0pt][0pt]{\ytableaushort{{2,\!3}}}}}$ means a subset $T_{2,1}=\{ 2,3 \}\subset [n]$ is assigned to the box $b_{2,1}$ of $\lambda$. 
We abbreviate $\scriptsize{{\raisebox{-3.0pt}[0pt][0pt]{\ytableaushort{{2,\!3}}}}}$ as $\scriptsize{{\raisebox{-3.0pt}[0pt][0pt]{\ytableaushort{{23}}}}}$ in this paper since we only use single-digit variables for set-valued tableaux. We write all tableaux $T\in{\rm SVT}((2,1),3)$ as follows:\\
\begin{align*}
&{\raisebox{-10pt}[0pt][0pt]{\ytableaushort{11,2}}}{\raisebox{-7pt}[0pt][0pt],\ } {\raisebox{-10pt}[0pt][0pt]{\ytableaushort{11,3}}}{\raisebox{-7pt}[0pt][0pt],\ }{\raisebox{-10pt}[0pt][0pt]{\ytableaushort{12,2}}}{\raisebox{-7pt}[0pt][0pt],\ }{\raisebox{-10pt}[0pt][0pt]{\ytableaushort{12,3}}}{\raisebox{-7pt}[0pt][0pt],\ }{\raisebox{-10pt}[0pt][0pt]{\ytableaushort{13,2}}}{\raisebox{-7pt}[0pt][0pt],\ }{\raisebox{-10pt}[0pt][0pt]{\ytableaushort{13,3}}}{\raisebox{-7pt}[0pt][0pt],\ }{\raisebox{-10pt}[0pt][0pt]{\ytableaushort{22,3}}}{\raisebox{-7pt}[0pt][0pt],\ }{\raisebox{-10pt}[0pt][0pt]{\ytableaushort{23,3}}}{\raisebox{-7pt}[0pt][0pt],\ }\\\\\\
&{\raisebox{-3pt}[0pt][0pt]{\ytableaushort{1{12},2}},\ }{\raisebox{-3pt}[0pt][0pt]{\ytableaushort{1{13},2}},\ }{\raisebox{-3pt}[0pt][0pt]{\ytableaushort{1{23},2}},\ } {\raisebox{-3pt}[0pt][0pt]{\ytableaushort{1{12},{3}}},\ }{\raisebox{-3pt}[0pt][0pt]{\ytableaushort{1{13},3}},\ }{\raisebox{-3pt}[0pt][0pt]{\ytableaushort{1{23},3}},\ }{\raisebox{-3pt}[0pt][0pt]{\ytableaushort{1{1},{23}}},\ }{\raisebox{-3pt}[0pt][0pt]{\ytableaushort{1{2},{23}}},\ }\\\\
&{\raisebox{-10pt}[0pt][0pt]{\ytableaushort{1{3},{23}}}}{\raisebox{-7pt}[0pt][0pt],\ }{\raisebox{-10pt}[0pt][0pt]{\ytableaushort{2{23},{3}}}}{\raisebox{-7pt}[0pt][0pt],\ }{\raisebox{-10pt}[0pt][0pt]{\ytableaushort{{12}2,3}}}{\raisebox{-7pt}[0pt][0pt],\ }{\raisebox{-10pt}[0pt][0pt]{\ytableaushort{{12}3,3}}}{\raisebox{-7pt}[0pt][0pt],\ }{\raisebox{-10pt}[0pt][0pt]{\ytableaushort{{1}{12},{23}}}}{\raisebox{-7pt}[0pt][0pt],\ }{\raisebox{-10pt}[0pt][0pt]{\ytableaushort{1{13},{23}}}}{\raisebox{-7pt}[0pt][0pt],\ }{\raisebox{-10pt}[0pt][0pt]{\ytableaushort{1{23},{23}}}}{\raisebox{-7pt}[0pt][0pt],\ }{\raisebox{-10pt}[0pt][0pt]{\ytableaushort{{12}{23},3}}}{\raisebox{-7pt}[0pt][0pt],\ }\\\\
&{\raisebox{-17pt}[0pt][0pt]{\ytableaushort{1{123},{2}}}}{\raisebox{-14pt}[0pt][0pt],\ }{\raisebox{-17pt}[0pt][0pt]{\ytableaushort{1{123},{3}}}}{\raisebox{-14pt}[0pt][0pt],\ }{\raisebox{-17pt}[0pt][0pt]{\ytableaushort{1{123},{23}}}}{\raisebox{-14pt}[0pt][0pt].\ }\\\\\\
\end{align*}
\end{case}

Also, we define the weight of $T$ as
\begin{equation}\label{wei}
\omega(T) = (\omega_1(T),\omega_2(T),\dots,\omega_n(T)) \in (\mathbb{Z}_{\geq 0})^n,
\end{equation}
where $\omega_i(T)$ denotes the number of $i$'s in $T$. Additionally, we define a monomial of $x=(x_1,x_2,\dots,x_n) \in \mathbb{C}^n$ having a weight $\omega(T)$ as
\begin{equation} \label{mono}
x^{\omega(T)} = x_1^{\omega_1(T)}x_2^{\omega_2(T)}\cdots x_n^{\omega_n(T)}.
\end{equation}
With these settings, following Buch~\cite{[Buc02]}, we define Grothendieck polynomials $G_\lambda$ as follows:
\begin{align}\label{GG}
G_{\lambda}=G_{\lambda}(x_1,x_2,\dots,x_n\mid\beta) = \sum_{T \in {\rm SVT}(\lambda,n)}\beta^{|T|-|\lambda|}x^{\omega(T)},
\end{align}
where $\beta$ is a parameter and $|T|$ is the number of assigned positive integers in $T$. 
Grothendieck polynomials $G_\lambda$ have the bi-alternant formula as follows~\cite{[IN13]} (see also~\cite{[Len00]}):
\begin{equation} \label{W}
G_\lambda = \frac{|x_i^{\lambda_j+n-j}(1+\beta x_i)^{j-1}|_{n \times n}}{\prod_{1\leq i < j \leq n}(x_i - x_j)}\eqqcolon\frac{p_\lambda}{\Delta_n}.
\end{equation}
Note that if it is necessary to add zeros to $\lambda$ such as
\begin{equation*}
\lambda = (\underbrace{\lambda_1,\lambda_2,\dots,\lambda_{\ell(\lambda)},0,\dots,0}_{n})\eqqcolon (\lambda_1,\dots,\lambda_n)
\end{equation*} 
for applying~(\ref{W}) to $G_\lambda$. We immediately obtain $s_\lambda$ by setting $\beta=0$ for~(\ref{GG}) or~(\ref{W}).

\section{Special values of Grothendieck polynomials in terms of hypergeometric functions}\label{S3}
\subsection{Special cases of Young diagrams}\label{Gauss}
In this subsection, we propose connections between the number of set-valued tableaux for Young diagrams with a single row or a single column and the hypergeometric function ${}_2F_1$. We represent the hypergeometric functions $_{2}F_{1}$ as follows:
\begin{align}\label{hgf}
_{2}F_{1} = {}_{2}F_{1}
\!\left( \hspace{-1mm} \begin{aligned}
& 
  \alpha,\ \beta \\
& 
   \ \ \gamma
\end{aligned} ; z \right) = \sum_{m=0}^\infty \frac{(\alpha)_m (\beta)_m}{(\gamma)_m}\frac{z^m}{m!}\quad (|z|<1), \\ \notag
\end{align}
where $(a)_m$ represents the Pochhammer symbol, defined as
\begin{align*}
(a)_m=a(a+1)(a+2)\cdots(a+m-1)\quad (a\in \mathbb{C},m \in \mathbb{Z}_{>0}).
\end{align*}
We adopt the convention $(a)_0=1$.
We find the following results.

\begin{prop}\label{p1}
For any $\lambda=(k)\ (k\in \mathbb{Z}_{>0})$, we have
\begin{align*}
G_{(k)}(1,1,\dots,1\mid \beta)&=\binom{n+k-1}{k}{}_{2}F_1\!\left( \hspace{-3mm} \begin{aligned}
& 
  \ \  k,\ 1-n \\
& 
   \ \ \ k+1
\end{aligned} ; -\beta \right).
\end{align*}
\end{prop}
\begin{proof}
We applied the Schur expansion of $G_{(k)}$ in~\cite{[Len00]} (see also Remark 3.7 of~\cite{[MPS21]}), we have
\begin{align*}
G_{(k)} = \sum_{m=0}^{n-1}\beta^{m}s_{(k,1^m)}.
\end{align*}
By putting $x_1=x_2=\cdots=x_n=1$, we obtain
\begin{align*}
G_{(k)}(1,1,\dots,1\mid \beta)&=\sum_{m=0}^{n-1}\beta^{m}\frac{(n)_m(1-n)_m(-1)^m}{(1)_{k-1}(1)_m(k+m)}\\
&=\binom{n+k-1}{k}\sum_{m=0}^{n-1}\frac{(k)_m(1-n)_m}{(k+1)_m}\frac{(-\beta)^m}{m!}\\
&=\binom{n+k-1}{k}{}_{2}F_1\!\left( \hspace{-3mm} \begin{aligned}
& 
  \ \  k,\ 1-n \\
& 
   \ \ \ k+1
\end{aligned} ; -\beta \right).
\end{align*}
Hence, we have the result.
\end{proof}
\begin{prop}\label{p2}
For any $\lambda=(1^k)\ (k=1,2,\dots,n)$, we have
\begin{align*}
G_{(1^k)}(1,1,\dots,1\mid \beta)&=\binom{n}{k}{}_{2}F_1\!\left( \hspace{-3mm} \begin{aligned}
& 
  \ \  k,\ k-n \\
& 
   \ \ \ k+1
\end{aligned} ; -\beta \right).
\end{align*}
\end{prop}
\begin{proof}
Following~\cite{[Len00]}, the $G_{(1^k)}$ can be expanded by $(m+k)$-th elementary symmetric polynomials $e_{m+k}=e_{m+k}(x)$:
\begin{align*}
G_{(1^k)}=\sum_{m=0}^{n-k}\beta^m\binom{n+k-1}{m}e_{m+k}.
\end{align*}
Assigning $x_1=x_2=\cdots=x_n=1$ leads to the following:
\begin{align*}
G_{(1^k)}(1,1,\dots,1\mid \beta)&=\sum_{m=0}^{n-k}\beta^m\binom{m+k-1}{m}\binom{n}{m+k}\\
&=\binom{n}{k}\sum_{m=0}^{n-k}\frac{(k)_m(k-n)_m}{(k+1)_m}\frac{(-\beta)^m}{m!}\\
&=\binom{n}{k}{}_{2}F_1\!\left( \hspace{-3mm} \begin{aligned}
& 
  \ \  k,\ k-n \\
& 
   \ \ \ k+1
\end{aligned} ; -\beta \right).
\end{align*}
Consequently, we obtain the result.
\end{proof}
As corollaries of these propositions, we immediately derive the following results.

\begin{cor}\label{sr}
For any $\lambda= (k)\ (k\in \mathbb{Z}_{\geq1})$, we have
\begin{align*}
|{\rm SVT}((k),n)| &=
\binom{n+k-1}{k}{}_{2}F_{1}
\!\left( \hspace{-3mm} \begin{aligned}
& 
  \ \  k,\ 1-n \\
& 
   \ \ \ k+1
\end{aligned} ; -1 \right).\\
\end{align*}
\end{cor}

\begin{cor}\label{sc}
For any $\lambda= (1^k)\ (k=1,2,\dots,n)$, we have
\begin{align*}
|{\rm SVT}((1^k),n)| &=
\binom{n}{k}{_2}F_{1}
\!\left( \hspace{-3mm} \begin{aligned}
& 
  \ \  k,\ k-n \\
& 
   \ \ \ k+1
\end{aligned} ; -1 \right).\\
\end{align*}
\end{cor}

\subsection{General cases of Young diagrams}\label{SS3.2}
In this subsection, we show special values of Grothendieck polynomials and provide an explicit formula for the number of set-valued tableaux in general of Young diagrams. Besides, we propose connections between the special values of Grothendieck polynomials and a multivariable hypergeometric function.

In~\cite{[CP21]}, \textit{refined Grothendieck polynomials} are introduced by generalizing the parameter $\beta$ of $G_\lambda$ to  $\mbox{\boldmath $\beta$}=(\beta_1,\beta_2,\dots,\beta_{j-1})$. These polynomials are defined using set-valued tableaux.
Additionally, following~\cite{[HJKSS21]}, \textit{refined canonical Grothendieck polynomials} are defined as:
\begin{align}\label{refG}
G_{\lambda}(x\mid \mbox{\boldmath $\alpha$},-\mbox{\boldmath $\beta$})=\frac{\left|x_i^{\lambda_j+n-j}\frac{(1-\beta_1 x_i)\cdots(1-\beta_{j-1}x_i)}{(1-\alpha_1 x_i)\cdots(1-\alpha_{\lambda_j}x_i)}\right|_{n \times n}}{\prod_{1\leq i < j \leq n}(x_i - x_j)}.
\end{align}
We set refined Grothendieck polynomials $G_{\lambda}(x\mid \mbox{\boldmath $\beta$})$ as follows and apply the principal generalization to derive an explicit formula for $|{\rm SVT}(\lambda,n)|$:
\begin{align}\label{refG_b}
G_{\lambda}(x\mid \mbox{\boldmath $\beta$})=\frac{\left| x_i^{\lambda_j+n-j}(1+\beta_1x_i)\cdots(1+\beta_{j-1}x_i)\right|}{\prod_{1\leq i < j \leq n}(x_i - x_j)}\eqqcolon\frac{p_\lambda{(x\mid \mbox{\boldmath $\beta$})}}{\Delta_n}.
\end{align}

\begin{thm}\label{M1}
For any $\lambda \vdash l$, we have
\begin{align}\label{eG}
&G_{\lambda}(1,q,q^2,\dots,q^{n-1}\mid \mbox{\boldmath $\beta$}) \notag \\
&= \sum_{k_1=0}^{0}\sum_{k_2 = 0}^{1} \cdots \sum_{k_{n} = 0}^{n - 1}e^{(0)}_{k_1}e^{(1)}_{k_2}\cdots e^{(n-1)}_{k_n}\prod_{1\leq i < j \leq n}\frac{q^{\lambda_j + n - j + k_j} - q^{\lambda_i + n - i + k_i}}{q^{n-j} - q^{n-i}},
\end{align}
where $e_{k_j}^{(j-1)}=e_{k_j}^{(j-1)}(\beta_1,\beta_2,\dots,\beta_{j-1})$ are $k_j$-th elementary symmetric polynomials and we promise $e^{(j-1)}_0=1$.
In particular, we obtain
\begin{align}\label{G_q}
&G_{\lambda}(1, q, q^2, \ldots, q^{n - 1}\mid \beta) \notag \\
&= \sum_{k_1=0}^{0}\sum_{k_2 = 0}^{1} \cdots \sum_{k_{n} = 0}^{n - 1}\binom{0}{k_1}\binom{1}{k_2}\cdots\binom{n - 1}{k_{n}}\beta^{k_1 + \cdots + k_{n}}\prod_{1\leq i < j \leq n}\frac{q^{\lambda_j + n - j + k_j} - q^{\lambda_i + n - i + k_i}}{q^{n-j} - q^{n-i}}.
\end{align}
\end{thm}

\begin{proof}
We calculate each column vector in $p_\lambda(x\mid \mbox{\boldmath $\beta$})$ as follows:
\begin{align*}
\left(
\begin{array}{c}
x_{1}^{\lambda_i + n - i}(1 + \beta_1 x_1)\cdots(1 + \beta_{j-1} x_1) \\
x_{2}^{\lambda_i + n - i}(1 + \beta_1 x_2)\cdots(1 + \beta_{j-1} x_2) \\
\vdots \\
x_{n}^{\lambda_i + n - i}(1 + \beta_1 x_n)\cdots(1 + \beta_{j-1} x_n) 
\end{array}
\right)
=
\sum_{k_{j} = 0}^{j - 1}e^{(j-1)}_{k_j}
\left(
\begin{array}{c}
x_{1}^{\lambda_i + n - i + k_{j}}\\
x_{2}^{\lambda_i + n - i + k_{j}}\\
\vdots \\
x_{n}^{\lambda_i + n - i + k_{j}}
\end{array}
\right).
\end{align*}
Thus, the determinant $p_\lambda(x\mid \mbox{\boldmath $\beta$})$ is transformed as
\begin{align}\label{q_1}
&p_\lambda(x\mid \mbox{\boldmath $\beta$}) \notag \\
=&\left|
\begin{array}{cccc}
x_{1}^{\lambda_1 + n - 1} & x_{1}^{\lambda_2 + n - 2}(1 + \beta_1 x_1) & \cdots & x_{1}^{\lambda_n}(1 + \beta_1 x_1)\cdots(1+\beta_{n-1} x_1) \\
x_{2}^{\lambda_1 + n - 1} & x_{2}^{\lambda_2 + n - 2}(1 + \beta_1 x_2) & \cdots & x_{2}^{\lambda_n}(1 + \beta_1 x_2)\cdots(1+\beta_{n-1} x_2)  \\
\vdots & \vdots & \  & \vdots \\
x_{n}^{\lambda_1 + n - 1} & x_{n}^{\lambda_2 + n - 2}(1 + \beta_1 x_n) & \cdots & x_{n}^{\lambda_n}(1 + \beta_1 x_n)\cdots(1+\beta_{n-1} x_n) 
\end{array}
\right| \notag \\
=& 
 \sum_{k_1 = 0}^{0}\sum_{k_2 = 0}^{1} \cdots \sum_{k_{n} = 0}^{n - 1}e^{(0)}_{k_1}e^{(1)}_{k_2}\cdots e^{(n-1)}_{k_n}
\left|
\begin{array}{cccc}
x_{1}^{\lambda_1 + n - 1 +k_1} & x_{1}^{\lambda_2 + n - 2 + k_{2}} & \cdots & x_{1}^{\lambda_n + k_{n}} \\
x_{2}^{\lambda_1 + n - 1 +k_1} & x_{2}^{\lambda_2 + n - 2 + k_{2}} & \cdots & x_{2}^{\lambda_n + k_{n}} \\
\vdots & \vdots & \  & \vdots \\
x_{n}^{\lambda_1 + n - 1 +k_1} & x_{n}^{\lambda_2 + n - 2 + k_{2}} & \cdots & x_{n}^{\lambda_n + k_{n}} 
\end{array}
\right|.
\end{align}
By putting $(x_1,x_2,x_3,\dots,x_n)=(1,q,q^2,\dots,q^{n-1})$, we have
\begin{align}\label{q_2}
&\left|
\begin{array}{cccc}
x_{1}^{\lambda_1 + n - 1 +k_1} & x_{1}^{\lambda_2 + n - 2 + k_{2}} & \cdots & x_{1}^{\lambda_n + k_{n}} \\
x_{2}^{\lambda_1 + n - 1 + k_1} & x_{2}^{\lambda_2 + n - 2 + k_{2}} & \cdots & x_{2}^{\lambda_n + k_{n}} \\
\vdots & \vdots & \ & \vdots \\
x_{n}^{\lambda_1 + n - 1 +k_1} & x_{n}^{\lambda_2 + n - 2 + k_{2}} & \cdots & x_{n}^{\lambda_n + k_{n}} 
\end{array}
\right| \notag \notag \\
=&
\left|
\begin{array}{cccc}
1 & 1 & \cdots & 1 \\
q^{\lambda_1 + n - 1 + k_1} & q^{\lambda_2 + n - 2 + k_{2}} & \cdots & q^{\lambda_n + k_{n}} \\
\vdots & \vdots & \ & \vdots \\
(q^{\lambda_1 + n - 1 + k_1})^{n - 1} & (q^{\lambda_2 + n - 2 + k_{2}})^{n - 1} & \cdots & (q^{\lambda_n + k_{n}})^{n - 1}
\end{array}
\right| \notag \\
=&
\prod_{1\leq i < j \leq n}(q^{\lambda_j + n - j + k_j} - q^{\lambda_i + n - i + k_i}).
\end{align}
On the other hand, we get
\begin{align}\label{q_3}
\Delta_n(1,q,q^2,\dots,q^{n-1}) = \prod_{1 \leq i<j \leq n} (q^{n-j} - q^{n-i}).
\end{align}
From~(\ref{q_1}),~(\ref{q_2}) and~(\ref{q_3}), we have the claim~(\ref{eG}).
Assigning $\beta_1=\beta_2=\beta_{j-1}=\beta$, the second claim~(\ref{G_q}) is immediately obtained.
\end{proof}
\begin{rem}
In~\cite{[Len00]}, the Schur expansions of $G_{\lambda}$ are given in general $\lambda$.
This formula is denoted as a combinatorial context.
It is not clear whether our result~(\ref{G_q}) can be derived from the combinatorial formula.
\end{rem}
In the following, we have an explicit formula of the number of set-valued tableaux by taking $q \rightarrow 1$ for $(\ref{G_q})$.
\begin{cor}\label{cor1}
For any $\lambda \vdash l$, we have
\begin{align}\label{c1}
|{\rm SVT}(\lambda, n)| = \sum_{k_1 = 0}^{0}\sum_{k_2 = 0}^{1} \cdots \sum_{k_{n} = 0}^{n - 1}\binom{0}{k_1}\binom{1}{k_2}\cdots\binom{n - 1}{k_{n}}\prod_{1\leq i < j \leq n}\frac{\lambda_i - \lambda_j + k_i - k_j +j-i}{j-i}.
\end{align}
\end{cor}

The equation~(\ref{c1}) can be denoted by using the following Holman's hypergeometric functions $F^{(n)}$:
\begin{align}\label{c2}
F^{(n)}=&F^{(n)}((A_{ij})_{(n-1)\times (n-1)}|(a_{ij})_{n \times u}|(b_{ij})_{n \times v}|(z_{i1})_{n\times 1}) \notag \\
=& \sum_{k_1,\dots,k_n=0}^{\infty}\left( \prod_{1\leq i<j\leq n}\frac{A_{ij}+k_i-k_j}{A_{ij}} \right)\left( \prod_{j=1}^{u}\prod_{i=1}^n(a_{ij})_{k_i} \right)\left( \prod_{j=1}^{v}\prod_{i=1}^n\frac{1}{(b_{ij})_{k_i}} \right)\left( \prod_{i=1}^n z_{i1}^{k_i} \right),
\end{align}
where
\begin{align*}
(A_{ij})_{(n-1)\times (n-1)} = \begin{pmatrix}
  A_{12} &\ & & \\
  A_{13} & A_{23} & & \text{\Huge 0} \\
  \vdots & \vdots &\ddots & \\
  A_{1n} & A_{2n} &\cdots & A_{n-1n}
\end{pmatrix}.
\end{align*}

\begin{thm}\label{M2}
For any $\lambda \vdash l$, we have
\begin{align}\label{G_H}
&G_{\lambda}(1, 1, \dots, 1\mid \beta) \notag \\
&= |{\rm SST}(\lambda,n)|\cdot F^{(n)} \left(\begin{pmatrix}
  A_{12} &\ & & \\
  A_{13} & A_{23} & & \text{\Huge $0$} \\
  \vdots & \vdots &\ddots & \\
  A_{1n} & A_{2n} &\cdots & A_{n-1n}
\end{pmatrix} \middle| \begin{pmatrix}
  0 \\
  -1 \\
  \vdots \\
  -(n-1)
\end{pmatrix} \middle| \begin{pmatrix}
  1 \\
  1 \\
  \vdots \\
  1
\end{pmatrix} \middle| \begin{pmatrix}
  -\beta \\
  -\beta \\
  \vdots \\
  -\beta
\end{pmatrix} \right),
\end{align}
where $A_{ij}=\lambda_i - \lambda_j + j - i$.

\end{thm}


\begin{proof}
We can rewrite the right-hand side of (\ref{G_q}) by considering $q \rightarrow 1$ as follows: 
\begin{align*}
&({\rm RHS\ of\ (\ref{G_q})}) \notag \\
&= \prod_{1 \leq i < j \leq n}\frac{A_{ij}}{j-i}\sum_{k_1,\dots,k_n=0}^{\infty} \prod_{1 \leq i < j \leq n}\frac{A_{ij}+k_i-k_j}{A_{ij}}\frac{(-0)_{k_1}(-\beta)^{k_1}}{(1)_{k_1}}\cdots \frac{(-(n-1))_{k_n}(-\beta)^{k_n}}{(1)_{k_n}} \notag \\
&=|{\rm SST}(\lambda,n)|\cdot \sum_{k_1,\dots,k_n=0}^{\infty}\left( \prod_{1\leq i <j \leq n}\frac{A_{ij}+k_i-k_j}{A_{ij}} \right)\left(\prod_{i=1}^n(-(i-1))_{k_i} \right)\left( \prod_{i=1}^n\frac{1}{(1)_{k_i}} \right)\left( \prod_{i=1}^n (-\beta)_{i1}^{k_i} \right)\notag \\
&= |{\rm SST}(\lambda,n)|\cdot F^{(n)} \left(\begin{pmatrix}
  A_{12} &\ & & \\
  A_{13} & A_{23} & & \text{\Huge $0$} \\
  \vdots & \vdots &\ddots & \\
  A_{1n} & A_{2n} &\cdots & A_{n-1n}
\end{pmatrix} \middle| \begin{pmatrix}
  0 \\
  -1 \\
  \vdots \\
  -(n-1)
\end{pmatrix} \middle| \begin{pmatrix}
  1 \\
  1 \\
  \vdots \\
  1
\end{pmatrix} \middle| \begin{pmatrix}
  -\beta \\
  -\beta \\
  \vdots \\
  -\beta
\end{pmatrix} \right).
\end{align*}
Subsequently, we have the result.
\end{proof}
\begin{rem}
It seems that the Grothendieck polynomials $G_\lambda(x\mid \beta)$ are not rewritten as a hypergeometric function like the equation~(\ref{G_H}) in general $x$.
\end{rem}
Due to the equation (\ref{G_H}), we immediately obtain the following result by setting $\beta=1$.

\begin{cor}\label{cor2}
For any $\lambda \vdash l$, we have
\begin{align}\label{Holman}
|{\rm SVT}(\lambda,n)| 
=|{\rm SST}(\lambda,n)|\cdot F^{(n)} \left(\begin{pmatrix}
  A_{12} &\ & & \\
  A_{13} & A_{23} & & \text{\Huge $0$} \\
  \vdots & \vdots &\ddots & \\
  A_{1n} & A_{2n} &\cdots & A_{n-1n}
\end{pmatrix} \middle| \begin{pmatrix}
  0 \\
  -1 \\
  \vdots \\
  -(n-1)
\end{pmatrix} \middle| \begin{pmatrix}
  1 \\
  1 \\
  \vdots \\
  1
\end{pmatrix} \middle| \begin{pmatrix}
  -1 \\
  -1 \\
  \vdots \\
  -1
\end{pmatrix} \right).
\end{align}
\end{cor}
In~\cite{FNS}, we proved the value of~(\ref{Holman}) is always odd. Furthermore, we established a special value of Grothendieck polynomials in the following.
\begin{prop}(\cite{FNS})\label{AA}
For any $\lambda \vdash l$, we have
\begin{align}\label{A1}
G_{\lambda}(\beta,\beta,\dots,\beta \mid {-\beta^{-1}})=\beta^{|\lambda|}.
\end{align}
\end{prop}
In particular, from~(\ref{A1}), we immediately obtain 
\begin{align}\label{1}
G_\lambda(1,1,\dots,1\mid -1)=1.
\end{align}
Using~(\ref{G_H}) and~(\ref{1}), we get the following result.
\begin{thm}\label{M3}
For any $\lambda \vdash l$, we have
\begin{align}\label{FSST}
F^{(n)} \left(\begin{pmatrix}
  A_{12} &\ & & \\
  A_{13} & A_{23} & & \text{\Huge $0$} \\
  \vdots & \vdots &\ddots & \\
  A_{1n} & A_{2n} &\cdots & A_{n-1n}
\end{pmatrix} \middle| \begin{pmatrix}
  0 \\
  -1 \\
  \vdots \\
  -(n-1)
\end{pmatrix} \middle| \begin{pmatrix}
  1 \\
  1 \\
  \vdots \\
  1
\end{pmatrix} \middle| \begin{pmatrix}
  1 \\
  1 \\
  \vdots \\
  1
\end{pmatrix} \right)=\frac{1}{|{\rm SST}(\lambda.n)|}.
\end{align}
\end{thm}
\begin{proof}
In~(\ref{G_H}), replace $\beta$ with $-\beta^{-1}$ and by specializing $\beta=1$, we obtain the following:
\begin{align*}
G(1,1,\dots,1 \mid -1) = |{\rm SST}(\lambda,n)|\cdot F^{(n)} \left(\begin{pmatrix}
  A_{12} &\ & & \\
  A_{13} & A_{23} & & \text{\Huge $0$} \\
  \vdots & \vdots &\ddots & \\
  A_{1n} & A_{2n} &\cdots & A_{n-1n}
\end{pmatrix} \middle| \begin{pmatrix}
  0 \\
  -1 \\
  \vdots \\
  -(n-1)
\end{pmatrix} \middle| \begin{pmatrix}
  1 \\
  1 \\
  \vdots \\
  1
\end{pmatrix} \middle| \begin{pmatrix}
  1 \\
  1 \\
  \vdots \\
  1
\end{pmatrix} \right).
\end{align*}
Hence, we have the result from (\ref{1}).
\end{proof}

\begin{rem}
In~\cite{Hol80,[Mil88]}, a summation formula for $F^{(n)}$ is given as follows:
\begin{align}\label{Holman2}
&F^{(n)} \left(\begin{pmatrix}
  A_{12} &\ & & \\
  A_{13} & A_{23} & & \text{\Huge $0$} \\
  \vdots & \vdots &\ddots & \\
  A_{1n} & A_{2n} &\cdots & A_{n-1n}
\end{pmatrix} \middle| \begin{pmatrix}
   a_{11} &\cdots &a_{nn+1} \\
   \vdots & \  & \vdots \\
   a_{nn+1} & \cdots & a_{nn+1}
\end{pmatrix} \middle| \begin{pmatrix}
   b_{11} &\cdots & b_{nn+1} \\
   \vdots & \  & \vdots \\
   b_{nn+1} & \cdots & b_{nn+1}
\end{pmatrix} \middle| \begin{pmatrix}
  1 \\
  1 \\
  \vdots \\
  1
\end{pmatrix} \right) \notag \\
&= \frac{\Gamma(b_{nn+1}-\sum_{i=1}^n a_{ii})}{\Gamma(b_{nn+1}-a_{nn+1})}\prod_{i=1}^n\frac{\Gamma(b_{in+1})}{\Gamma(b_{in+1}-a_{ii})},
\end{align}
where $n\in \mathbb{Z}_{\geq 1}$ and
\begin{align*}
{\rm Re}\left(  \sum_{i=1}^{n+1}b_{ki}-\sum_{i=1}^{n+1}a_{ki}\right)>n\quad (1\leq k \leq n),
\end{align*}
and, the parameters satisfy the following conditions:
\begin{equation}\label{q_4}
\begin{split}
A_{id}-A_{ic}&=A_{cd} \quad(c < d),\\
a_{id}-a_{cr}&=A_{ic} \quad (i < c),\\
b_{id}-b_{cr}&=A_{ic} \quad (i < c),\\
b_{ii}&=1 \quad(1\leq i \leq n).
\end{split}
\end{equation}
By focusing on $(a_{ij})$ and $(b_{ij})$, our summation formula~(\ref{FSST}) is different from~(\ref{Holman2}). For instance, for $\lambda=(2,1)$ and $n=3$, we specialize $(a_{ij})$, $(b_{ij})$ and $(A_{ij})$ for $F^{(3)}$ as follows:
\begin{align}\label{special}
(a_{ij})_{3\times 1}=\begin{pmatrix}
  0  \\
 -1 \\
 -2
\end{pmatrix}, \quad (b_{ij})_{3\times 1}=\begin{pmatrix}
  1 \\
  1 \\
  1
\end{pmatrix},
\end{align}
\begin{align*}
(A_{ij})_{2\times 2}=\begin{pmatrix}
  A_{12} & 0 \\
  A_{13} & A_{23}
\end{pmatrix}=\begin{pmatrix}
  \lambda_1-\lambda_2+2-1 & 0 \\
  \lambda_1-\lambda_3+3-1 & \lambda_2-\lambda_3+3-2 \\
\end{pmatrix}=\begin{pmatrix}
  2 & 0 \\
  4 & 2 \\
\end{pmatrix}.
\end{align*}
From these matrices and putting $z_{11}=z_{21}=\cdots=z_{n1}=1$, we obtain the subsequent value using~(\ref{FSST}).
\begin{align*}
&F^{(3)} \left(\begin{pmatrix}
  A_{12} & 0 \\
  A_{13} & A_{23}
\end{pmatrix} \middle| \begin{pmatrix}
  0 \\
  -1 \\
  -2
\end{pmatrix} \middle| \begin{pmatrix}
  1 \\
  1 \\
  1
\end{pmatrix} \middle| \begin{pmatrix}
  1 \\
  1 \\
  1
\end{pmatrix} \right)\\
&=\sum_{k_1=0}^0\sum_{k_2=0}^1\sum_{k_3=0}^2\frac{A_{12}+k_1-k_2}{A_{12}}\frac{A_{13}+k_1-k_3}{A_{13}}\frac{A_{23}+k_2-k_3}{A_{23}}\frac{(0)_{k_1}(-1)_{k_2}(-2)_{k_3}}{(1)_{k_1}(1)_{k_2}(1)_{k_3}}\\
&=\frac{2}{2\cdot 4\cdot 2}=\frac{1}{8}.
\end{align*}
We find that the specialization~(\ref{special}) does not satisfy the second and the third conditions of~(\ref{q_4}).

\end{rem}

\section{Conclusions}\label{S4}
In this paper, we gave some special values of Grothendieck polynomials $G_\lambda$. For single row and single column $\lambda$, the value $G_{\lambda}(1,\dots,1\mid \beta)$ is expressed by the Gauss' hypergeometric functions ${}_2F_1$ (Proposition \ref{p1}, \ref{p2}). Also, the value $G_\lambda(1,\dots,1 \mid \beta)$ for general $\lambda$ is given by the Holman's hypergeometric functions $F^{(n)}$ (Theorem~\ref{M2}). Theorem ~\ref{M2} is derived by considering the principal specialization of the Grothendieck polynomials (Theorem~\ref{M1}). Using these special values, we gave explicit formulas for the number of set-valued tableaux for shapes $\lambda$ with $n$ variables (Corollary~\ref{sr},~\ref{sc} and~\ref{cor1}).
As an application, we obtained a summation formula for  $F^{(n)}$ (Theorem~\ref{M3}).


There are various generalizations of Grothendieck polynomials $G_\lambda$. For example, canonical Grothendieck polynomials with another parameter $\alpha$ were introduced by Yeliussizov~\cite{[Yel17]}.
The refined canonical Grothendieck polynomial~(\ref{refG}) is an extension of these polynomials. 
Since (refined) canonical Grothendieck polynomials are defined as a bi-alternant formula similar to~(\ref{W}), we expect our results can be extended to these polynomials.


Besides, there is no bi-alternant formula for skew Grothendieck polynomials $G_{\lambda/\mu}$~\cite{[Buc02]} which are defined as a tableaux sum formula in a similar manner to~(\ref{GG}). 
If a bi-alternant formula for $G_{\lambda/\mu}$ is obtained, then our results can be also generalized.
Furthermore, there is another skew version of Grothendieck polynomials $G_{\lambda/\!/\mu}$~\cite{[Buc02]}.
These polynomials can be often better behaved than the combinatorial description.
For example, we can find some details about $G_{\lambda/\!/\mu}$ in~\cite{[IKS22]}.
It is interesting to generalize our results for $G_{\lambda/\mu}$ and $G_{\lambda/\!/\mu}$. 

Finally, the special values of Grothendieck polynomials $G_\lambda$ are given by Holman's hypergeometric functions $F^{(n)}$. This function $F^{(n)}$ was introduced in the representation theory of Lie group $U(n+1)$ and $SU(n+1)$.
It is interesting and important to interrupt $G_\lambda$ in the context of $U(n+1)$ or $SU(n+1)$.

\section*{Acknowledgments} 

We would like to thank Yasuhiko Yamada for helpful discussions and constant encouragement. 
We deeply appreciate Hirofumi Yamada and Shunya Adachi for useful discussions and motivating us.
We are also thankful to Travis Scrimshaw for giving informative comments about Proposition~\ref{AA}, which is the key point of Theorem~\ref{M3}.
This work was supported by JST SPRING, Grant Number JPMJSP2148, and JSPS KAKENHI Grant Number 22H01116.

\noindent{\sc Department of Mathematics, Graduate School of Science, Kobe University}

(Taikei Fujii) {\it E-mail address}: {\tt tfujii@math.kobe-u.ac.jp}

(Takahiko Nobukawa) {\it E-mail address}:  {\tt tnobukw@math.kobe-u.ac.jp}

(Tatsushi Shimazaki) {\it E-mail address}: {\tt tsimazak@math.kobe-u.ac.jp}


\begin{thebibliography}{LS82}
\bibitem[1]{[Buc02]} A. S. Buch, \textit{A Littlewood-Richardson rule for the K-theory of Grassmannians}, Acta. Math.,
{\bf 189}(1): 37$\mathchar`-$78, 2002.
\bibitem[2]{[CP21]}M. Chan, N. Pflueger, \textit{Combinatorial relations on skew Schur and skew stable Grothendieck polynomials}, Algebr. Comb., {\bf 4}(1):  175-188, 2021.
\bibitem[3]{FNS}T. Fujii, T. Nobukawa and T. Shimazaki \textit{The number of the set-valued tableaux is odd}, arXiv: 2305.06740, 2023.
\bibitem[4]{Hol80}W. J. Holman III, \textit{Summation Theorems for Hypergeometric Series in U(n)}, SIAM J. Math. Anal. {\bf 11}(3): 523-532, 1980.
\bibitem[5]{[HJKSS21]}B. H. Hwang, J. Jang, J. S. Kim, M. Song, M, and U. Song, \textit{Refined canonical stable Grothendieck polynomials and their duals}, arXiv: 2104.04251, 2021.
\bibitem[6]{[IN13]} T. Ikeda and H. Naruse, \textit{K-theoretic analogues of factorial Schur P-and Q-functions}, In: Adv. Math., {\bf 243}: 22$\mathchar`-$66, 2013.
\bibitem[7]{[IKS22]}S. Iwao, K. Motegi, and T. Scrimshaw, \textit{Free-fermions and canonical Grothendieck polynomials}, arXiv: 2211.05002, 2022.
\bibitem[8]{[Las90]} A. Lascoux, \textit{Anneau de Grothendieck de la vari\'{e}t\'{e} de drapeaux}, The Grothendieck Festschrift, Vol. {\bf III}: Progr. Math., Birkh\"{a}user, Boston, 1$\mathchar`-$34, 1990.
\bibitem[9]{[LS82]} A. Lascoux and M. P. Sch\"{u}tzenberger, \textit{Structure de Hopf de l'anneau de cohomologie et de
l'anneau de Grothendieck d'une vari\'{e}t\'{e} de drapeaux}, C. R. Acad. Sci. Paris S\'{e}r. I Math, {\bf 295}(11): 629$\mathchar`-$633,
1982.
\bibitem[10]{[Len00]} C. Lenart, \textit{Combinatorial aspects of the K-theory of Grassmannians}, Ann. Comb. {\bf 4}: 67-82, 2000.
\bibitem[11]{[Mac95]} I. G. Macdonald, \textit{Symmetric Functions and Hall Polynomials}, 2nd. ed., Oxford, 1995.
\bibitem[12]{[Mil88]} S. C. Milne, \textit{A q-analog of the Gauss summation theorem for hypergeometric series in U(n)}, Adv. in Math., {\bf 72}(1): 59-131, 1988.
\bibitem[13]{[MPS21]}C. Monical, O. Pechenik, and T. Scrimshaw. \textit{Crystal structures for symmetric Grothendieck polynomials}, Transform. Groups {\bf 26}: 1025-1075, 2021.
\bibitem[14]{[Nou23]} M. Noumi, \textit{Macdonald polynomials Commuting family of q-difference Operators and Their Joint Eigenfunctions}, SpringerBriefs in Mathematical Physics, {\bf 50}: Springer, Singapore, 2023.
\bibitem[15]{[Yel17]} D. Yeliussizov, \textit{Duality and deformations of stable Grothendieck polynomials}, J. Algebraic Comb, {\bf 45}(1): 295-344, 2017.





\end{thebibliography}
\end{document}